\documentclass[a4paper, reqno,12pt]{amsart}

\usepackage{tikz, tikz-3dplot}
\usepackage{amsmath, amsthm, amssymb,graphics }

\theoremstyle{plain}
\newtheorem{te}{Theorem}[section]
\newtheorem{lem}[te]{Lemma}
\newtheorem{co}[te]{Corollary}
\newtheorem{pr}[te]{Proposition}
\newtheorem{de}[te]{Definition}

\theoremstyle{remark}
\newtheorem{re}[te]{Remark}
\newtheorem*{ack*}{Acknowledgment}

\def\0{{\bf 0}}

\def\R{{\mathbb R}}

\def\C{{\mathbb C}}

\def\Z{{\mathbb Z}}

\def\V{{\mathbb V}}

\def\sp{{\operatorname{span}}}

\def\supp{{\operatorname{supp}\,}}
\def\nint{\mathop{\diagup\kern-13.0pt\int}}

\def\Ic{{\mathcal I}}\def\Fc{{\mathcal F}}
\def\Cc{{\mathcal C}}\def\Dc{{\mathcal D}}

\def\Sc{{\mathcal S}}\def\Vc{{\mathcal V}}

\begin{document}

\title[The Walsh--Fourier transform]{The failure of the fractal uncertainty principle for the Walsh--Fourier transform}

\author{Ciprian Demeter}
\address{Department of Mathematics, Indiana University, 831 East 3rd St., Bloomington IN 47405}
\email{demeterc@indiana.edu}

\keywords{}
\thanks{The author is  partially supported by the  Research NSF grant DMS-1800305}

\dedicatory{To the memory of Jean Bourgain}
\begin{abstract}
We construct $\delta$-regular sets with $\delta\ge \frac12$ for which the analog of the Bourgain--Dyatlov Fractal Uncertainty Principle fails for the Walsh--Fourier transform.
\end{abstract}
\maketitle

\section{The Fractal Uncertainty principle for the Fourier transform}
\bigskip

This note explores the so-called Fractal Uncertainty principle, one of the last signifficant results of Jean Bourgain. The principle is a fundamental result in Fourier analysis with far-reaching consequences  in the spectral theory of hyperbolic surfaces.

\begin{de}
	\label{l9}
Let $X\subset \R$ be a nonempty closed set. Consider the constants $\delta\in [0,1)$, $C_R\ge 1$ and $0\le \alpha_0\le \alpha_1\le \infty$. We say that $X$ is $\delta$-regular with constant $C_R$ on scales $\alpha_0$ to $\alpha_1$ if there is a Borel measure $\mu_X$ supported on $X$ such that
\\
\\
$\bullet$ for each interval $I$ of size $|I|\in [\alpha_0,\alpha_1]$, we have $\mu_X(I)\le C_R|I|^\delta$
\\
\\
$\bullet$ if additionally $I$ is centered at a point in $X$, then $\mu_X(I)\ge C_R^{-1}|I|^{\delta}$.
\end{de}
We will denote by $|X|$ the Lebesgue measure of $X$.
\smallskip

Examples of regular sets will be discussed in Section \ref{s3}. At this point, we only mention that $\delta$-regular sets need to have small Lebesgue measure, more precisely (see Lemma 2.9 in \cite{BD2})
\begin{equation}
\label{l7}
|X|\le 24C_R^2\alpha_1^\delta\alpha_0^{1-\delta}.
\end{equation}

The following Fractal Uncertainty principle for the Fourier transform
$$\widehat{f}(\xi)=\int_\R f(x)e^{-2\pi i x\xi}dx$$
was proved in \cite{BD2}. It refines earlier versions due to Dyatlov-Zahl \cite{DZ} and Bourgain--Dyatlov \cite{BD1}.

\begin{te}
	\label{BD}
Let $\delta\in [0,1)$, $C_R\ge 1$ and $N\ge 1$. Assume that
\\
\\
$\bullet$ $X\subset [0,1]$ is $\delta$-regular with constant $C_R$ on scales $\frac1N$ to 1
\\
$\bullet$ $Y\subset [0,N]$ is $\delta$-regular with constant $C_R$ on scales $1$ to N.
\\
\\
Then there exist constants $\beta>0$ and $C$, both depending only on $\delta$ and  $C_R$, such that for each $f\in L^2(\R)$ with Fourier transform  supported on $Y$ we have
\begin{equation}
\label{l10}
\|f\|_{L^2(X)}\le CN^{-\beta}\|f\|_{L^2(\R)}.
\end{equation}
\end{te}

When $\delta<\frac12$, this theorem has an easy proof that also provides an explicit value for $\beta$. For reader's convenience, we recall this argument below.
 If $\widehat{f}$ is supported on $Y$ we have
\begin{align*}
\|f\|_{L^2(X)}&\le |X|^{1/2}\|f\|_{L^\infty(\R)}\\&\le |X|^{1/2}\|\widehat{f}\|_{L^1(\R)}\\&=|X|^{1/2}\|\widehat{f}\|_{L^1(Y)}\\&\le |X|^{1/2}|Y|^{1/2}\|\widehat{f}\|_{L^2(\R)}\\&= |X|^{1/2}|Y|^{1/2}\|{f}\|_{L^2(\R)}.
\end{align*}
If $X$ and $Y$ are as in the theorem, then \eqref{l7} implies that $|X|^{1/2}|Y|^{1/2}\le CN^{-\beta}$, $\beta=\frac12-\delta$.
\medskip

On the other hand, the proof from \cite{BD2} in the case $\delta\ge \frac12$ is very involved. At its heart, it relies both on the multi-scale structure of  regular sets, and  on the following unique continuation result (Lemma 3.2 in \cite{BD2}).
\begin{lem}
\label{l8}
Let $\Ic$ be a non overlapping collection of intervals of size 1 and let $c_0>0$. For each $I\in\Ic$, let $I''\subset I$ be an interval of size $c_0$. Then there exists a constant $C$ depending only on $c_0$ such that for all $r\in(0,1)$, $0<\kappa\le e^{-C/r}$ and $f\in L^2(\R)$ with $\widehat{f}$ compactly supported, we have
$$\sum_{I\in\Ic}\|f\|_{L^2(I)}^2\le \frac{C}{r}(\sum_{I\in\Ic}\|f\|_{L^2(I'')}^2)^{\kappa}\|e^{2\pi r|\xi|}\widehat{f}(\xi)\|_{L^2(\R)}^{2(1-\kappa)}.$$
\end{lem}	

In the next section we recall the details about the Walsh transform, a closely related, though technically simpler analog of the Fourier transform. We will construct sets $X$ and $Y$ as in Theorem \ref{BD} with regularity $\delta\ge \frac12$, such that the Fractal Uncertainty Principle fails when the Walsh transform replaces the Fourier transform. This fundamental difference between the behavior of the two transforms explains why the proof in \cite{BD2} is so complicated. The argument in \cite{BD2} must necessarily rely not just on the fine  structure of the regular sets, but also on the stronger form of the Uncertainty Principle that governs the Fourier world. This has to do with the fact that there is no (nontrivial) compactly supported function whose Fourier transform is also compactly supported. Lemma \ref{l8} is a manifestation of this principle.

In the next section we will see that there are compactly supported $L^2$ functions whose Walsh transforms are also compactly supported. This easily shows the failure of Lemma \ref{l8}, and ultimately of Theorem \ref{BD}, in the Walsh framework. Our main result, Theorem \ref{main1} is proved in the last section.
\smallskip

\begin{ack*} I am grateful to Semyon Dyatlov for pointing out to me the following facts. 
The Walsh transform  appears in applications to the toy model of open quantum baker's maps. In that case the maps can be Walsh-quantized  and the spectral gap results do sometimes fail on these. This was observed by Nonnenmacher and Zworski in  \cite{NZ} (Section 5, in particular Remark 5.2). In  \cite{DJ} (at the end of the introduction) Dyatlov and  Jin  briefly interpreted this phenomenon as special instances of the failure of the fractal uncertainty principle  for the Walsh--Fourier transform.
\end{ack*}

\section{The Walsh transform}

For more details on the material in this section, the reader may consult the original paper of Walsh \cite{Wal},  or the modern reference \cite{Sch}.

Let $\Z_2=\{0,1\}$ with addition modulo 2 and Haar measure splitting the mass evenly between $0$ and $1$.
We consider the infinite product group $G=\prod_{1}^\infty\Z_2$ equipped with the product Haar measure. This is sometimes referred to as the {\em Cantor group}.
\medskip

Let $\Dc=\{j2^{-i}:\;0\le j\le 2^{i}\}$ be the dyadic numbers in $[0,1]$. They have zero Lebesgue measure. The map
$$\Phi:G\to [0,1],\;\;\;\Phi(a_{-1},a_{-2},\ldots)=\sum_{k\le -1} a_{k}2^{k}$$
is almost bijective -- if $x\in[0,1]\setminus\Dc$,  $\Phi^{-1}(\{x\})$ consists of one point -- measurable, and maps the Haar measure on  $G$ to the Lebesgue measure $|\cdot|$ on [0,1]. This suggests a  natural way to identify $G$ with $([0,1],\oplus, |\cdot|)$, where  $\oplus$ is defined as follows. Given $x,y\in [0,1]\setminus\Dc$, $x=\sum_{k\le -1}x_k2^{k}$, $y=\sum_{k\le -1}y_k2^{k}$, we write
$$x\oplus y=\sum_{k\le -1}c_k2^{k},\;\;\;c_k=x_k+y_k\pmod 2.$$
See Sec 2.2 in \cite{Fol} for details.

\medskip

The characters on $G$ are the so-called Walsh functions. For $n\ge 0$  the $n-$th Walsh function $W_n:[0,1)\to\{-1,1\}$ is defined recursively by the formula
$$W_0=1_{[0,1)}$$
$$W_{2n}(x)=W_n(2x)+W_n(2x-1)$$
$$W_{2n+1}(x)=W_n(2x)-W_n(2x-1).$$
In particular,
$$W_1(x)=\begin{cases}1,\;&0\le x<\frac12\\-1,\;&\frac12\le x<1 \end{cases},$$
$$W_2(x)=\begin{cases}1,\;&x\in[0,\frac14)\cup[\frac12,\frac34)\\-1,\;&x\in[\frac14,\frac12)\cup[\frac34,1)\end{cases}$$
$$W_3(x)=\begin{cases}1,\;&x\in[0,\frac14)\cup[\frac34,1),\\-1,\;&x\in[\frac14,\frac34)\end{cases}.$$
\medskip

In many ways, the functions $W_n$ resemble the (Fourier) system of  exponentials $e^{2\pi inx}$.
For example, the functions $(W_n)_{n\ge 0}$ form an orthonormal basis for $L^2([0,1])$.
See Sec 4.1 \cite{Fol} for more details.
\medskip

The Walsh--Fourier coefficients of a function $f:[0,1]\to\C$ are given by
$$\Fc_Wf(n)=\int f(x)W_n(x)dx,\;\;n\ge 0.$$
To get a greater perspective on the role of the Walsh system and its closeness to the Fourier system of exponentials,
we introduce a new operation. For $x,y\in[0,\infty)$ having unique representations (that is, for Lebesgue almost all pairs $(x,y)$)
$$x=\sum_{k=-\infty}^\infty x_k2^{k},\;\;\; y=\sum_{k=-\infty}^\infty y_k2^{k},$$
we define
 $$x\otimes y:=\sum_{k=-\infty}^\infty c_k2^{k}$$
where
$$c_k=\sum_{j\in\Z}x_jy_{k-j}\pmod 2.$$
We note that this  sum is always finite. From now on, we will implicitly ignore the zero measure dyadic points.
\medskip

Define the function $e_W:[0,\infty)\to\{-1,1\}$ such that $e_W(x)=1$ when $x_{-1}=0$ and $e_W(x)=-1$ when $x_{-1}=1$. This  1-periodic function is the Walsh analogue of $e^{2\pi ix}$. It is easy to check that
\begin{equation}
\label{l11}
W_n(x)=e_W(x\otimes n)1_{[0,1]}(x).
\end{equation}

We may introduce the Walsh (also called Walsh-Fourier) transform of a compactly supported function $f:[0,\infty)\to\C$ to be the function
$$\Fc_Wf:[0,\infty)\to\C,\;\;\;\Fc_Wf(y):=\int_{[0,\infty)} e_W(x\otimes y)f(x)dx.$$
The Walsh--Fourier inversion formula takes the form $\Fc_W\circ\Fc_W=id$. \medskip

It is worth noting that
$$e_W(x\otimes y)=e_W(x\otimes z)$$
whenever $x\in[0,1)$ and $n\le y,z<n+1$. Consequently,  if $f$ is supported on $[0,1]$ then $\Fc_Wf$ is constant on intervals $[n,n+1)$. This explains why for such functions the Walsh--Fourier coefficients completely characterize the function $f$.
\medskip

While the Walsh transform behaves very similar to the Fourier transform, it has one notable feature that makes it easier to work with. This has to do with the fact that there are (plenty of) compactly supported functions whose Walsh transforms are also compactly supported.   A quick computation shows that for each dyadic interval $I=[l2^{k},(l+1)2^{k})$ we have
\begin{equation}
\label{l1}
\Fc_W{1_{I}}(y)=|I|1_{[0,|I|^{-1}]}(y)e(x_I\otimes y),
\end{equation}
where $x_I$ is an arbitrary element of $I$. Because of this feature,  typically the results that hold in the Fourier case are expected to also hold in the Walsh setting, with the argument in the latter case being cleaner, less technical. The approach of first proving results in the Walsh setting and then ``transferring" them to the Fourier world was successfully employed in the time-frequency analysis of modulation invariant operators, starting with \cite{Thi}. The interested reader may consult the survey paper \cite{Dem}, which explores a few different arguments for the Walsh analog of Carleson's Theorem and contains some relevant references.
\medskip

In this paper we present an example that goes against the aforementioned philosophy. We show that a  fundamental result that holds for the Fourier transform is in fact false for the Walsh transform.

\section{The main result}
\label{s3}

The ``textbook" example of regular sets can be constructed as follows. Fix integers $0<M<L$. Let $\Sc$ be a collection of subsets $S$ of $\{0,1,\ldots, L-1\}$ with cardinality $M$. We create a collection of nested sets $X_1,X_2,\ldots$ as follows. Pick $S_1\in\Sc$ and let
$$A_1=L^{-1}S_1,\;\;\;X_1=A_1+[0,L^{-1}].$$
Next, for each $a\in A_1$, choose some $S_{2,a}\in\Sc$ and define
$$A_{2,a}=a+L^{-2}S_{2,a},\;\;\;A_2=\cup_{a\in A_1}A_{2,a},\;\;\;X_2=A_2+[0,L^{-2}].$$
The rest of the construction is recursive. Assume we have constructed $A_j$ and $X_j$ for $1\le j\le n-1$. For each $a\in A_{n-1}$, choose some $S_{n,a}\in\Sc$ and define
$$A_{n,a}=a+L^{-n}S_{n,a},\;\;\;A_n=\cup_{a\in A_{n-1}}A_{n,a},\;\;\;X_n=A_n+[0,L^{-n}].$$
Note that $X_n\subset [0,1]$ consists of $M^n$ intervals $I\in\Ic_{X_n}$ of length $L^{-n}$. Also,  $X_n$ is $\frac{\log M}{\log L}-$regular on scales $\frac1{L^n}$ to $1$, with constant $C_{n}$ satisfying the uniform bound $C_{n}\le C(M,L)$, where $C(M,L)$ depends only on $M,L$. The reader may check that Definition \ref{l9} is satisfied with the measure $\mu_{X_n}$ given by $\mu_{X_n}(I)=\frac1{M^n}$, for each $I\in\Ic_{X_n}$.

 \medskip

We specialize this construction  as follows. Fix the positive integers $m_1$ and $m_2\ge m_1$. We consider a set as above with $M=2^{m_2}$ and $L=2^{m_1+m_2}$. The collection $\Sc$ will consist of only the set $S=\{k2^{m_1},\;0\le k\le 2^{m_2}-1\}$.

More precisely, define
$$A_n=\{\sum_{i=1}^n\frac{k_{n-i+1}2^{m_1}}{L^i}:\;0\le k_1,\ldots,k_n\le 2^{m_2}-1\}$$
and
\begin{equation}
\label{l3}
X_n=A_n+[0,L^{-n}].
\end{equation}
Then $X_n\subset[0,1]$ is  $\frac{m_2}{m_1+m_2}-$regular on scales $L^{-n}$ to $1$, with constant $C_{n}$ uniformly bounded in $n$.

Define also the dilate
$$Y_n=L^nX_n=\{L^nx:\;x\in X_n\}.$$
Note that $Y_n$ is the union of intervals of length $1$ and $Y_n\subset [0,L^n]$. It is $\frac{m_2}{m_1+m_2}-$regular on scales $1$ to $L^n$, with the same constant $C_{n}$ as $X_n$.

\begin{te}
	\label{main1}
The (real) vector space $\Vc_{X_n,Y_n}$ of all $L^2$ functions
$$f:[0,1]\to\R,\;\;\supp f\subset X_n,\;\;\supp \Fc_W f\subset Y_n$$
has dimension at least $2^{n(m_2-m_1)}$. In particular, for each $n\ge 1$ there is a function $f_n$ (other than the zero function) with $\Fc_Wf_n$ supported on $Y_n$ such that
$$\|f_n\|_{L^2(X_n)}=\|f_n\|_{L^2([0,1])}.$$

\end{te}
Fixing $m_1,m_2$ and letting $n\to\infty$ shows that the Walsh analog of \eqref{l10} fails to hold for any $\beta>0$, when $\delta\ge \frac12$.

\medskip

We remark that the restriction $m_2\ge m_1$ is needed in Theorem \ref{main1}, as it is equivalent with the lower bound $\delta\ge \frac12$ for the regularity of $X_n,Y_n$. When $\delta<\frac{1}2$, Theorem \ref{BD} remains true in the Walsh framework and the argument from the first section for the Fourier case translates to the Walsh case, too.

\section{proofs}
We start by proving a sequence of lemmas.

\begin{lem}
For $x,y\in[0,\infty)$ and $l\in \Z$ we have
$$(2^l x)\otimes y=x\otimes(2^ly).$$
\end{lem}
\begin{proof}
If $$x=\sum_{k\in \Z}x_k2^k,\;\;\;y=\sum_{k\in \Z}y_k2^k$$
then
$$2^lx=\sum_{k\in \Z}x_{k-l}2^k,\;\;\;2^ly=\sum_{k\in \Z}y_{k-l}2^k$$
and
$$((2^l x)\otimes y)_k=\sum_{j\in\Z}(2^lx)_{j}y_{k-j}=\sum_{j\in\Z}x_{j-l}y_{k-j}=\sum_{j\in\Z}x_{j}y_{k-j-l}=(x\otimes(2^ly))_k.$$
\end{proof}
Combining this lemma with \eqref{l11}   and \eqref{l1} reveals that if $I=[\frac{k}{L^n},\frac{k+1}{L^n}]\subset[0,1]$ then
\begin{equation}
\label{l2}
\Fc_W1_I(y)=L^{-n}W_k(\frac{y}{L^n}).
\end{equation}

\begin{lem}
The functions $W_0,W_1,\ldots,W_{2^m-1}$ span the vector space
$$\Cc_{m}=\{f:[0,1]\to\R:\;f\text{ constant on dyadic intervals of length }2^{-m}\}.$$
\end{lem}
\begin{proof}
An easy induction argument based on the recursive formula for $W_n$ shows that $W_0,W_1,\ldots,W_{2^m-1}\in\Cc_m$. The vector space $\Cc_m$ has dimension $2^m$, and since $W_0,W_1,\ldots,W_{2^m-1}$ are linearly independent (being orthogonal), they form a basis for this space.

\end{proof}
The recursive definition of $W_n$ also immediately implies the following periodicity property.
\begin{lem}
The function $W_{k2^l}$ is $2^{-l}$ periodic, if $k,l$ are positive integers. Moreover, when $x\in [0,2^{-l}]$ we have
$$W_{k2^l}(x)=W_k(x2^{l}).$$
\end{lem}
The combination of the last two  lemmas yields the following result.
\medskip

\begin{pr}
Consider the (real) vector space of all $F:[0,1]\to\R$ having the following two properties for some positive integers $l,m$
\\
\\
(P1): $F$ is $2^{-l}$ periodic
\\
\\
(P2): $F$ is constant on dyadic intervals of length $2^{-l-m}$.
\\
\\
Then this vector space coincides with the span of the Walsh functions $W_{k2^l}$, for $0\le k\le 2^m-1$.
\end{pr}
\medskip

Let us recall that $L=2^{m_1+m_2}$. Rescaling the above result gives:
\begin{co}
For $1\le i\le n$, consider the (real) vector space $\Vc_{i,n}$ of all functions   $F_i:[0,L^n]\to\R$ such that
\\
\\
(P1): $F_i$ is $\frac{L^i}{2^{m_1}}$ periodic
\\
\\
(P2): $F_i$ is constant on dyadic intervals of length $L^{i-1}$.
\\
\\
Then $\Vc_{i,n}$ coincides with the span of the rescaled Walsh functions $W_{kL^{n-i}2^{m_1}}(\frac{y}{L^n})$, for $0\le k\le 2^{m_2}-1$.
\end{co}
\medskip

Let $\Vc_{X_n}$ be the (real) vector space  spanned by the Walsh transforms $\Fc_W1_I$ of all intervals $I$ of length $L^{-n}$ in $X_n$.
According to \eqref{l3} and \eqref{l2} this is the same as the vector space spanned by the rescaled Walsh functions
\begin{equation}
\label{l4}
W_{\sum_{i=1}^n{k_{n-i+1}2^{m_1}}{L^{n-i}}}(\frac{y}{L^n}):\;0\le k_1,\ldots,k_n\le 2^{m_2}-1.
\end{equation}
Note that  $\Vc_{X_n}$ is a  proper subset of the family of Walsh transforms of functions supported on $X_n$.  We are going to search for functions in $\Vc_{X_n}$ that are supported on $Y_n$.

\begin{lem}
For each $k,k'\in\Z$	
$$W_{k}W_{k'}=W_{k\oplus k'}.$$
\end{lem}
\begin{proof}
$$W_{k}(x)W_{k'}(x)=(-1)^{(x\oplus k)_{-1}}(-1)^{(x\oplus k')_{-1}}=(-1)^{\sum_{j}x_jk_{-1-j}}(-1)^{\sum_{j}x_jk'_{-1-j}}$$$$=(-1)^{\sum_jx_j(k\oplus k')_{-1-j}}=W_{k\oplus k'}(x).$$
\end{proof}
Combining the last lemma and corollary we get:

\begin{pr}
	\label{l5}
The space $\Vc_{X_n}$ coincides with the collection of arbitrary finite  sums of products (over $i$) of  functions $F_i\in\Vc_{i,n}$.
\end{pr}
\begin{proof}
Note that since $k_{n-i+1}2^{m_1}<L$ we have
$${\sum_{i=1}^n{k_{n-i+1}2^{m_1}}{L^{n-i}}}={\oplus_{i=1}^n{k_{n-i+1}2^{m_1}}{L^{n-i}}},$$
where the factors on the right hand side  are summed using $\oplus$ rather than $+$. Thus
$$W_{\sum_{i=1}^n{k_{n-i+1}2^{m_1}}{L^{n-i}}}(\frac{y}{L^n})=\prod_{i=1}^n W_{k_{n-i+1}2^{m_1}L^{n-i}}(\frac{y}{L^n}).$$
\end{proof}
\bigskip

\begin{lem}
\label{hgrfhghsxm,xn,mcn}
Assume that $f_1,\ldots,f_N:\R\to\R$ are linearly independent, $g_1,\ldots,g_M:\R\to\R$ are linearly independent and $\{f_ng_m:\;1\le n\le N,\;1\le m\le M\}$ are linearly independent. Let $\V_1$ be a linear subspace of $\sp(f_1,\ldots,f_N)$ of dimension $d_1\le N$, and let $\V_2$ be a linear subspace of $\sp(g_1,\ldots,g_M)$ of dimension $d_2\le M$. 

Then the linear space $\V$ spanned by the functions $fg$ with $f\in \V_1$ and $g\in\V_2$ has dimension $d_1d_2$.
\end{lem}
\begin{proof}
It is clear that $\dim(\V)\le \dim(\V_1)\dim(\V_2)$, so it remains to prove the reverse inequality. Assume that the functions $f^{(k)}=\sum_{n=1}^N{a_{k,n}f_n}$ with $1\le k \le d_1$ form a basis for $\V_1$. It follows that the $(d_1,N)$ matrix $A=(a_{k,n})$ contains a nonsingular $(d_1,d_1)$ minor $A'$.    Assume that the functions $g^{(l)}=\sum_{m=1}^M{b_{l,m}g_m}$ with $1\le l \le d_2$ form a basis for $\V_2$. It follows that the $(d_2,M)$ matrix $B=(b_{l,m})$ contains a nonsingular $(d_2,d_2)$ minor $B'$.

We will show that the functions $f^{(k)}g^{(l)}$, $1\le k\le d_1$, $1\le l\le d_2$ are linearly independent. We order the functions $f_ng_m$ using the lexicographic order for pairs $(n,m)$, that is $f_{1}g_1,\ldots, f_1g_M,f_2g_1,\ldots,f_2g_M,\ldots, f_Ng_1,\ldots,f_Ng_M$. We similarly order the functions $f^{(k)}g^{(l)}$ lexicographically with respect to the pairs $(k,l)$. We construct the $(d_1\times d_2,N\times M)$ matrix $C$ as follows. The $(i,j)$ entry is the coefficient of the $i^{th}$ function $f^{(k)}g^{(l)}$ with respect to the $j^{th}$ function $f_ng_m$. We denote this matrix by $A\otimes B$. One easy way to visualize it is to start with the matrix $A$ and replace each entry $a_{k,n}$ with the matrix $a_{k,n}B$
$$C=\begin{bmatrix} a_{1,1}B&\ldots&a_{1,N}B\\\ldots&\ldots&\ldots\\a_{d_1,1}B&\ldots&a_{d_1,N}B\end{bmatrix}.$$ 

We need to prove that $C$ contains a nonsingular $(d_1 d_2,d_1d_2)$ minor. We claim that this minor is $C'=A'\otimes B'$, with the tensor operation described above. It is immediate that $C'$ is a minor of $C$. Also, it is well known that 
$$\det(C')=\det(A')^{d_1}\det(B')^{d_2}.$$
See for example \cite{Si}. In particular, $\det(C')\not=0$, as desired.  

\end{proof}

We now prove Theorem \ref{main1} by induction. It suffices to show that the vector space of those $F$ supported on $Y_n$, that are in the span of the rescaled Walsh functions in \eqref{l4}, has dimension at least $2^{n(m_2-m_1)}$.

Let us start with the base case $n=1.$ Using the characterization from Proposition \ref{l5}, it suffices to prove  that the vector space
$$\{F\in\Vc_{1,1}:\;\supp F\subset Y_1\}$$
has dimension $2^{m_2-m_1}$. The functions $F$ in this space are $2^{m_2}$ periodic and constant on all intervals $[l,l+1)$. Since $Y_1$ contains exactly $2^{m_2-m_1}$ unit intervals in $[0,2^{m_2}]$ (these are  $I_k=[k2^{m_1},k2^{m_1}+1]$, $0\le k\le 2^{m_2-m_1}-1$), and since
$$Y_1=\bigcup_{0\le k'\le 2^{m_1}-1}((I_0\cup I_1\cup\ldots\cup I_{2^{m_2-m_1}-1})+k'2^{m_2}),$$
it is immediate that the values of $F$ on $I_0, \ldots, I_{2^{m_2-m_1}-1}$ may be chosen arbitrarily. This verifies the base case of the induction. Note that by choosing the  values of $F$ to be 1 on all these intervals, we get the function $F=1_{Y_1}$. This shows that $\Fc_W1_{Y_1}$ is in $ \Vc_{X_1,Y_1}$.\medskip

Next, let us prove the theorem for $n\ge 2$, assuming its validity for $n-1$. We write
\begin{equation}
\label{l6}
Y_n=LY_{n-1}\cap Z_n,\;\;\;Z_n=\bigcup_{k\le \frac{L^n}{2^{m_1}}}[k2^{m_1},k2^{m_1}+1].
\end{equation}

Let $\Vc_{1,n}(Z_n)$ be the vector space  of those $F_1\in \Vc_{1,n}$ that are supported on $Z_n$. Note first that this has dimension $2^{m_2-m_1}$, since there are $2^{m_2-m_1}$ unit intervals in $Z_n$ that lie in the periodicity interval $[0,2^{m_2}]$ associated with $\Vc_{1,n}$. Pick $2^{m_2-m_1}$ functions $H$ in the span of $W_{{k_n}L^{n-1}2^{m_1}}$, with $0\le {k_n}\le 2^{m_2}-1$, such that the rescaled functions $H(\frac{y}{L^n})$ form a basis for $\Vc_{1,n}(Z_n)$.
\smallskip

By the induction hypothesis, we may find a subset consisting of $2^{(n-1)(m_2-m_1)}$ linear independent functions $G$ in the span of
$$W_{\sum_{i=1}^{n-1}{k_{n-i}2^{m_1}}{L^{n-1-i}}}:\;0\le k_1,\ldots,k_{n-1}\le 2^{m_2}-1$$
such that each
$G(\frac{y}{L^{n-1}})$
is supported on $Y_{n-1}$. So $G(\frac{y}{L^n})$ is supported on $LY_{n-1}$.

Because of Lemma \ref{hgrfhghsxm,xn,mcn}, \eqref{l6} and since
$$W_{\sum_{i=1}^n{k_{n-i+1}2^{m_1}}{L^{n-i}}}(\frac{y}{L^n})=W_{{k_n}L^{n-1}2^{m_1}}(\frac{y}{L^n})W_{\sum_{i=1}^{n-1}{k_{n-i}2^{m_1}}{L^{n-1-i}}}(\frac{y}{L^n})$$
is supported on $Y_n$, we conclude that there are at least $2^{n(m_2-m_1)}$ linearly independent functions in $\Vc_{X_n}$ (recall that these are functions spanned by the functions in \eqref{l4}) that are supported on $Y_n$. We thus have
 $$\dim \Vc_{X_n,Y_n}\ge 2^{n(m_2-m_1)}.$$
\bigskip

 \begin{re}
 The inductive argument from above shows that in fact $F=\Fc_W1_{Y_n}$ is in $ \Vc_{X_n,Y_n}$. Indeed, we observed that this is true for $n=1$. The case $n>1$ follows since 
 $$1_{Y_n}=1_{LY_{n-1}}1_{Z_n}$$
 and since $1_{Z_n}\in \Vc_{1,n}(Z_n).$
 \end{re}

\end{document}